\documentclass[12pt]{amsart}
\usepackage{amssymb,amsmath,enumitem}
\usepackage{graphicx, xcolor}
\newcommand{\R}{\mathbb{R}}

\newcommand{\M}{\mathcal{M}}

\newcommand{\be}{\begin{equation}}
\newcommand{\ee}{\end{equation}}
\newcommand{\bee}{\begin{equation*}}
\newcommand{\eee}{\end{equation*}}
\newcommand{\bea}{\begin{eqnarray}}
\newcommand{\eea}{\end{eqnarray}}
\newcommand{\bess}{\begin{eqnarray*}}
\newcommand{\eess}{\end{eqnarray*}}

\numberwithin{equation}{section}
\usepackage{hyperref}
\usepackage{lineno} 
\hypersetup{colorlinks=true, citecolor=teal, linkcolor=magenta, urlcolor=blue}
\theoremstyle{plain}
\newtheorem{Thm}{Theorem}[section]
\newtheorem{Cor}[Thm]{Corollary}
\newtheorem{Lem}[Thm]{Lemma}

\theoremstyle{definition}
\newtheorem{Def}[Thm]{Definition}

\theoremstyle{remark}

\begin{document}
\title[Solutions to Sublinear Elliptic Problems]{Minimal $L^p$-Solutions to Singular \\ Sublinear Elliptic Problems}   
\author{Aye Chan May}
\address[Aye Chan May]{School of Integrated Science and Innovation,  Sirindhorn International Institute of Technology,  Thammasat University, Thailand}
\email{\href{m6422040748@g.siit.tu.ac.th}{m6422040748@g.siit.tu.ac.th}}

\author{Adisak Seesanea}
\address[Adisak Seesanea]{School of Integrated Science and Innovation,  Sirindhorn International Institute of Technology,  Thammasat University, Thailand}
\email{\href{adisak.see@siit.tu.ac.th}{adisak.see@siit.tu.ac.th}}

\begin{abstract}
We solve the existence problem for the minimal positive solutions $u\in L^{p}(\Omega, dx)$ to the Dirichlet problems
for sublinear elliptic equations of the form
\[
\begin{cases} 
\mathcal{L}u=\sigma u^q+\mu\qquad \quad  \text{in} \quad \Omega, \\
\liminf\limits_{x \rightarrow y}u(x) = 0 \qquad y \in \partial_{\infty}\Omega,
\end{cases}
\]
where $0<q<1$ and $\mathcal{L}u:=-\text{div} (\mathcal{A}(x)\nabla u)$ is a linear uniformly elliptic operator with bounded measurable coefficients.  The coefficient $\sigma$ and data $\mu$ are nonnegative Radon measures on an arbitrary domain $\Omega \subset \R^n$ with a positive Green function associated with $\mathcal{L}$.
Our techniques are based on the use of sharp Green potential pointwise estimates, weighted norm inqualities, and norm estimates in terms of generalized energy.
\end{abstract}


\maketitle
\tableofcontents

\section{Introduction}

Let $\Omega$ be a nonempty open connected set in $\R^n$ ($n \geq 3$) which possesses a positive Green function $G$, and $\mathcal{M}^+(\Omega)$ denotes the class of all nonnegative Radon measures in $\Omega.$  

We consider the Dirichlet problem 
\be \label{classicallaplacian}
\begin{cases} 
         \mathcal{L}u=\sigma u^q +\mu,\quad     u \geq 0 \quad \text{in}\quad \Omega,\\		
		\liminf \limits_{x \rightarrow  y}u(x)=0,\quad y \in \partial_{\infty} \Omega
\end{cases} 
\ee
in the sublinear case $0<q<1$ where $\sigma, \mu \in \mathcal{M}^+(\Omega)$.  

Here $\mathcal{L}u:=-\text{div} (\mathcal{A}(x)\nabla u)$  with bounded measurable coefficients is assumed to be uniformly elliptic, i.e., $\mathcal{A}:\Omega\to\R^{n\times n}$ is a real symmetric matrix-valued function and there exists positive constants $m \leq M$ so that 
\[
m |\xi|^{2} \leq \mathcal{A}(x)\xi \cdot \xi \leq M |\xi|^{2}
\]
for almost every  $x \in \Omega$ and for every $ \xi \in\mathbb{R}^n.$

In this paper, a solution $u$ to the problem \eqref{classicallaplacian} will be understood in the sense that $u$ is an $\mathcal{A}$-superharmonic function on $\Omega$ such that $u\in L^q_{loc}(\Omega,d\sigma)$ with $u\geq 0\; d\sigma$-a.e., and satisfies the corresponding integral equations
\be \label{integraleq}
u = \mathbf{G}(u^{q} d\sigma) + \mathbf{G}\mu  \quad \text{in} \;\;  \Omega.
\ee

Here, the Green potential of a measure $\sigma\in \mathcal{M}^+(\Omega),$ is defined by 
\[
\mathbf{G}\sigma=\int_\Omega G(x,y)d\sigma(y),\quad x\in\Omega
\]
where a function $G:\Omega\times\Omega\to(0,\infty]$ called a positive Green function associated with $\mathcal{L}$ in $\Omega.$

In the classical case $\mathcal{L} :=-\Delta $, these sublinear equations are closely 
related to the study of porous medium equations, and  were studied by Brezis and Kamin \cite{BK} under the assumption of the bounded domain. The reader can also see such a sublinear problem under various assumptions, for instance, \cite{BB, BBC, CV2, QV, SV2, SV1, SV3, V4}, and the literature cited there.

There are many of the existing solutions theories to elliptic equations \eqref{classicallaplacian} involving measures. For instance, V\'{e}ron \cite{Ve}, considered problem \eqref{classicallaplacian} with different boundary conditions: homogeneous Dirichlet boundary conditions ($u=0$ on $\partial\Omega$) and measure boundary conditions ($u=\mu$ on $\partial\Omega$ where $\mu$ is Radon measure) under a smooth bounded domain $\Omega.$


The homogeneous case ($\mu= 0$) of problem \eqref{classicallaplacian} was investigated by Seesanea and Verbitsky \cite{SV3}. Nevertheless, when it comes to
the case $\mu \geq 0$, the relation between $\sigma$ and $\mu$
seems to be nontrivial in the scale of Lebesgue space.

Furthermore, in \cite{V4}, the author introduced the bilateral pointwise estimates in terms of the intrinsic nonlinear potentials 
which can be utilized to obtain the existence of a positive solution $u \in L^{p}(\Omega, dx)$ to \eqref{classicallaplacian}. Unfortunately, the definition of intrinsic nonlinear potential is  defined in terms of the best localized constant of related sublinaer weighted norm inequality,   which make it difficult to be verified.

In this present paper, we aims to provide a simple approach to overcome the difficulties in \cite{SV3} and deduce useful sufficient conditions 
on measures $\sigma$ and $\mu$ for the existence of the positive minimal $\mathcal{A}$-superharmonic solution $u \in L^{p}(\Omega, dx)$ to \eqref{classicallaplacian}. 

Our main results read as follows.
\begin{Thm} \label{thm:classical}
Let $\sigma, \mu \in \mathcal{M^+}(\Omega)$  such that $(\sigma,\mu) \neq (0,0)$, $0<q<1$ and $G$ be a positive Green function associated with $\mathcal{L}$ in $\Omega\subset\mathbb{R}^n, n \geq 3$. Suppose also that $\frac{n}{n-2}<p<\infty$, 
\be \label{greenpotentialwithsigma}
\mathbf{G}\sigma\in L^{\frac{\gamma+q}{1-q}}(\Omega,d\sigma)
\ee
and 
\be \label{greenpotentialwithmu}
\mathbf{G}\mu\in L^{\gamma}(\Omega,d\mu),
\ee
 with $\gamma=\frac{p(n-2)-n}{n}.$ Then there exists a positive minimal $\mathcal{A}$-superharmonic solution $u \in L^{p}(\Omega,dx)$ to \eqref{classicallaplacian}.
\end{Thm}
A sufficient condition for \eqref{greenpotentialwithsigma} and \eqref{greenpotentialwithmu} with $\gamma=\frac{p(n-2)-n}{n}$ is given by 
\be \label{suffconsigma}
\sigma\in L^{s_1}(\Omega,dx),\quad s_1=\frac{np}{n(1-q)+2p}
\ee
and
\be \label{suffconmu}
\mu\in L^{s_2}(\Omega,dx),\quad s_2=\frac{np}{n+2p}
\ee
where $\frac{n}{n-2}<p<\infty.$ Therefore, the following corollary can be simply deduced from Theorem  \ref{thm:classical}.

\begin{Cor}\label{corollary}
Under the assumptions of Theorem \ref{thm:classical}, if conditions \eqref{suffconsigma} and  \eqref{suffconmu} are fulfilled, then there exists a positive minimal $\mathcal{A}$-superharmonic solution  $u\in L^p(\Omega,dx)$ to \eqref{classicallaplacian}.
\end{Cor}

Observe that Corollary \ref{corollary} was done by Boccardo and Orsina \cite{BO}, with a different proof, when $\Omega$ is a bounded domain in $\R^n$.
\\
\\
\textit{Organization of the paper}
\\

In Section \ref{sec2}, we organize some definitions and well-known research that are relevant to our problem. In Section \ref{sec3}, we prove an estimate for $p$-th integrability of potentials in terms of generalized Dirichlet energy and the existence result of the problem by applying the previous estimate. Moreover, we provide a sufficient condition for the existence of a positive solution to \eqref{classicallaplacian}.
\\
\\
\textit{Notation}
\\

We use the following notation in this paper. Let $\Omega$ be a connected open subset in $\mathbb{R}^n$.
\begin{itemize}
\item $D$:= a relatively compact open subset of $\Omega$.
\item $\mathcal{H}_\mathcal{A}(D)$:= the set of all continuous $\mathcal{A}$-harmonic functions in $D.$
 \item $\mathcal{C}_0^{\infty}(\Omega)$:= the set of all smooth compactly  supported functions on $\Omega$.
 \item $\mathcal{M}^+(\Omega)$:= the set of all nonnegative Radon measures on $\Omega$.
  \item $L^p(\Omega, d\mu)$:= the $L^p$ space with respect to Radon measure $\mu\in\mathcal{M}^+(\Omega).$
  \item $L^p(\Omega, dx)$:= the $L^p$ space with respect to Lebesgue measure. 
\end{itemize}


 \section{Preliminaries} \label{sec2}
Thoughout, let $\Omega$ be a domain (connected open set) in $\mathbb{R}^n$.
\subsection{Function Spaces}
\begin{Def} 
For $1\leq p<\infty$ and $\mu \in \mathcal{M}^+(\Omega),$ we denote by $L^p(\Omega,d\mu)$ the space of all real-valued measurable functions $f$ on $\Omega$ such that
\[
\|f\|_{L^p(\Omega,d\mu)}=\big(\int_\Omega |f(x)|^p d\mu(x)\big)^\frac{1}{p}<\infty.
\]
\end{Def}
\begin{Def}
A function $u\in W^{1,2}_{loc}(\Omega)$ is said to be {\bf $\mathcal{A}$-harmonic} if $u$ satisfies the equation
\[
\mathcal{L}u=0\quad\text{in}\quad{\Omega}
\]
in the distributional sense, i.e.,
\[
\int_\Omega \mathcal{A}(x,\nabla u(x)) \cdot \nabla\phi dx=0,\quad\forall \phi\in \mathcal{C}^\infty_0(\Omega).
\]
\end{Def}
The set of {$\mathcal{A}$}-harmonic functions on $\Omega$ is denoted by $\mathcal{H}_{\mathcal{A}}(\Omega)$. Every $\mathcal{A}$-harmonic function $u$ has a continuous representative which coincides with $u$ a.e. see\cite[Theorem 3.70] {KHM}.  
\begin{Def}
A function $u:\Omega\to(-\infty,+\infty]$ is {\bf $\mathcal{A}$-superharmonic} if $u$ is lower semicontinuous in $\Omega,$ $u\not\equiv +\infty$ in each component of $\Omega,$ and for each $D\Subset\Omega$ and $h \in \mathcal{C}(\bar{D})\cap\mathcal{H}_{\mathcal{A}}(D)$, the inequality $u\geq h$ on $\partial D$ implies $u\geq h$ in $D.$
\end{Def} 

Let $u$ be an $\mathcal{A}$-superharmonic function in $\Omega.$ Then there exists a unique measure $\omega\in\mathcal{M}^+(\Omega)$ such that
\[
\mathcal{L}u=\omega\quad\text{in}\quad\Omega
\]
in the distributional sense, i.e.,
\[
\int_\Omega \mathcal{A}(x,\nabla u(x)) \cdot \nabla\phi \;dx=\int_\Omega \phi \;d\omega,\quad\forall \phi\in \mathcal{C}^\infty_0(\Omega).
\]
The measure $\omega$ is called the Riesz measure associated with $u,$ see \cite[Theorem 21.2]{KHM}.

\subsection{Potentials}\label{subsec:00000}
Let $\gamma>0$ and $\omega\in\mathcal{M}^+(\Omega),$ and let $G$ be a positive Green function associated with $\mathcal{L}$ on $\Omega.$ 
The generalized Green energy, introduced in \cite{SV2}, of a mesure $\omega \in \mathcal{M}^{+}{\Omega}$ is given by
 \[
 \mathcal{E}_\gamma[\omega]:=\int_\Omega (\mathbf{G}\omega)^\gamma d\omega.
 \]

The first theorem gives an auxiliary fact that will be used in the proof of the main lemma. The complete proof can be seen in \cite[Lemma 3.3]{SV2}.
\begin{Thm}[{\text See \cite{SV2}}]\label{risezmeasure} 
	Let $0<\gamma<1$ and $\mu\in\mathcal{M^+}(\Omega)$.  Suppose $G$ is a positive Green function associated with $\mathcal{L}$ on $\Omega$. Suppose $u:\mathbf{G}\mu\not\equiv \infty.$ Then $w:=u^\gamma$ is a positive $\mathcal{A}$-superharmonic function on $\Omega,$ and $w=\mathbf{G}\omega,$ where $\omega\in\mathcal{M}^+(\Omega)$ is the Riesz measure of $w$. Moreover,
	\[
	 \mathcal{E}_\gamma[\mu] <+\infty
	 \quad\text{if and only if}\quad
	 \mathcal{E}_{1}[\omega] <+\infty
\]
\end{Thm}

The next theorem provides sharp lower pointwise estimates for supersolutions to sublinear elliptic equations due to Grigor'yan and Verbitsky, \cite[Theorem 1.3]{GV}.
\begin{Thm}[ {\text See \cite{GV}}]\label{thm:pointwise} 
	Let $0<q<1$ and $\sigma\in\mathcal{M^+}(\Omega)$.  Suppose $G$ is a positive Green function associated with $\mathcal{L}$ on $\Omega$. If $u\in L^q_{loc}(\Omega,d\sigma)$ is a positive supersolution  to the sublinear integral equation
	\be \label{subint}
	u\geq \mathbf{G}(u^qd\sigma),\quad x\in\Omega,
	\ee
	then
	\be \label{globlow}
	u(x)\geq (1-q)^{\frac{1}{1-q}}\big[\mathbf{G}\sigma(x)\big]^{\frac{1}{1-q}},\quad x\in\Omega.
	\ee 
\end{Thm}
We use the following pointwise iterated inequalities to derive the Green potential estimate,  see \cite[ Lemma 2.5]{GV}.

\begin{Thm}[{\text See \cite{GV}}]\label{thm:iterated} 
Let $\sigma\in\mathcal{M^+}(\Omega)$ with $\sigma \not\equiv 0$, and let $G$ be the positive Green function associated with $\mathcal{L}$ on $\Omega.$ Then the following estimates hold.\\
(i) If $t\geq 1,$ then
\be \label{iterated}
(\mathbf{G}\sigma)^t(x)\leq t \mathbf{G}((\mathbf{G}\sigma)^{t-1}d\sigma)(x),\quad   x\in \Omega. 
\ee
(ii) If $0<t\leq1,$ then
\be \label{iteratedgeq}
(\mathbf{G}\sigma)^t(x)\geq t\mathbf{G}((\mathbf{G}\sigma)^{t-1}d\sigma)(x),\quad x\in \Omega.
\ee
\end{Thm}

The argument of finding a solution depends on the following weighted norm inequalities of the $(s,r)$-type in the case where $0<r<s$ and $1<s<\infty,$ for operators $\mathbf{G}$:
\be \label{weightednorm}
\|\mathbf{G}(fd\sigma)\|_{L^r(\Omega,d\sigma)}\leq c\|f\|_{L^s(\Omega,d\sigma)},\quad f\in L^s(\Omega,d\sigma),
\ee
where $c$ is a positive constant independent of $f,$ for an arbitrary measure $\sigma \in \mathcal{M^+}(\Omega),$ under certain assumptions on $G,$ see \cite[Theorem 1.1]{V1}.
\begin{Thm} [{\text See \cite{V1}}] \label{thm:weightednorm}  
Let $\sigma\in\mathcal{M^+}(\Omega)$ with $\sigma \not\equiv 0,$ and let $G$ be the positive Green function associated with $\mathcal{L}$ on $\Omega.$\\
(i)If $1<s<\infty$ and $0<r<s,$ then the weighted norm inequality \eqref{weightednorm} is fulfilled if and only if
\be \label{greenpotsr}
\mathbf{G}\sigma \in L^{\frac{sr}{s-r}}(\Omega,d\sigma).
\ee
(ii)If $0<q<1$ and $0<\gamma<\infty,$ then there exists a positive (super)solution $u\in L^{\gamma +q}(\Omega,d\sigma)$ to sublinear integral equation \eqref{subint} if and only if the weighted norm inequality \eqref{weightednorm} is fulfilled with $r=\gamma+q$ and $s=\frac{\gamma+q}{q},$ i.e.,
\be \label{weightednormth}
\|\mathbf{G}(fd\sigma)\|_{L^{\gamma+q}(\Omega,d\sigma)}\leq c\|f\|_{L^{\frac{\gamma+q}{q}}(\Omega,d\sigma)},\quad  f\in L^{\frac{\gamma+q}{q}}(\Omega,d\sigma),
\ee
or equivalently,
\be \label{greenpotrq}
\mathbf{G}\sigma\in L^{\frac{\gamma+q}{1-q}}(\Omega,d\sigma).
\ee

\end{Thm}

The next theorem gives the existence of a positive solution $u\in L^{\gamma+q}(\Omega,d\sigma)(\gamma>0)$ to the integral equation \eqref{integraleq} in the sublinear case, which is claimed by Seesanea and Verbitsky, see \cite[Theorem 4.2]{SV2}.

\begin{Thm} [\text See \cite{SV2}]  \label{thm:solofinteq}  
	Let $0<q<1, 0<\gamma<\infty$ and $\sigma,\mu \in \mathcal{M^+}(\Omega)$ with $\sigma,\mu \not\equiv 0.$ Suppose $G$ is a positive quasi-symmetric lower semicontinuous kernel on $\Omega \times \Omega,$ which satisfies the WMP. If \eqref{greenpotentialwithsigma} and
	\be \label{greenpotwithsigma}
	\mathbf{G}\mu\in L^{\gamma+q}(\Omega,d\sigma)
	\ee  hold, then there exists a positive (minimal) solution $u\in L^{\gamma+q}(\Omega,d\sigma)$ to \eqref{integraleq}.
	The converse statement is valid without the quasi-symmetry assumption on $G.$
\end{Thm}
The following lemma was stated in \cite[Lemma 4.3]{SV2}. We can control the interaction between the measure coefficient and measure data by applying the following lemma.
\begin{Lem}[\text See \cite{SV2}] \label{lemma}
	Let $0<q<1, 0<\gamma<\infty,$ and $\sigma,\mu \in \mathcal{M^+}(\Omega).$ Suppose $G$ is a positive Green function associated with $\mathcal{L}$ on $\Omega.$ Then conditions \eqref{greenpotentialwithsigma} and \eqref{greenpotentialwithmu} imply \eqref{greenpotwithsigma}.
\end{Lem}
The next results are essential lemmas to prove positive solutions to \eqref{classicallaplacian} when $\mu=0.$
The complete proofs of the following two lemmas can be found in \cite[Lemma 4.1 and Lemma 4.2]{SV3}.
\begin{Lem} [\text See \cite{SV3}]   \label{weightedlemma}
Let $0<q<1,$ and let $\sigma\in\mathcal{M}^+(\Omega)$ with $\sigma\not\equiv 0.$ Suppose $G$ is a positive Green function associated with $\mathcal{L}$ on $\Omega.$ Suppose that $\frac{n}{n-2}<p<\infty$ and the condition 
\be \label{neccond}
\mathbf{G}\sigma\in L^{\frac{p}{1-q}}(\Omega,dx)
 \ee
 is valid. Then
\[
\|\mathbf{G}(gd\sigma)\|_{L^p(\Omega,dx)}\leq c\|\mathbf{G}\sigma\|^{\frac{1}{s'}}_{L^\frac{p}{1-q}(\Omega,dx)} \|f\|_{L^s(\Omega,d\sigma)},\quad f\in L^s(\Omega,d\sigma),
\]
where $c$ is a positive constant independent of $f$ and $q$ and $s=\frac{p(n-2)-n(1-q)}{nq}.$
\end{Lem}

\begin{Lem} [\text See \cite{SV3}] \label{lemma2}
Let $0<q<1,$ and let $\sigma\in\mathcal{M}^+(\Omega)$ with $\sigma\not\equiv 0.$ Suppose $G$ is a positive Green function associated with $\mathcal{L}$ on $\Omega$. If $\frac{n}{n-2}<p<\infty,$ then \eqref{greenpotentialwithsigma} with $\gamma=\frac{p(n-2)-n}{n}$ implies \eqref{neccond}. In fact, 
\[
\|\mathbf{G}\sigma\|_{L^{\frac{p}{1-q}}(\Omega,dx)}\leq \tilde{C} \big\|\mathbf{G}\sigma\big\|_{L^{\frac{\gamma+q}{1-q}}(\Omega,d\sigma)}^{\frac{\gamma+q}{\gamma+1}}
\]
where $\tilde{C}$ is a positive constant depending on $\gamma$ and $q$.
\end{Lem}
 \section{Construction of minimal $L^p$-solutions} \label{sec3}
 In this section, we prove our main result stated Theorem \ref{thm:classical} and its
 consequence in Corollary \ref{corollary}.
 
The following lemma is one of the key ingredients in our approach.
 \begin{Lem}\label{mainlemma}
 Let $0<\gamma<\infty$ and $0<q<1.$ Suppose $G$ is a positive Green function associated with $\mathcal{L}$ in $\Omega \subset  \R^n, n\geq 3$. Let $\mu\in \mathcal{M}^+(\Omega)$ such that $\mathbf{G}\mu\not\equiv \infty.$ If $\omega\in\mathcal{M}^+(\Omega)$ is a Riesz measure of the $\mathcal{A}$-superharmonic function $(\mathbf{G}\mu)^{1-q},$ then 
 \[
 \mathcal{E}_{\frac{\gamma+q}{1-q}}[\omega]\leq C  \mathcal{E}_\gamma [\mu]
 \] 
 where $C$ is a positive constant depending on $\gamma$ and $q.$
 \end{Lem}
 \begin{proof}
 We have $w:=(\mathbf{G}\mu)^{1-q}$ is a positive $\mathcal{A}$-superharmonic function in $\Omega$, and $w:=\mathbf{G}\omega$ where $\omega\in\mathcal{M}^+(\Omega)$ is a Riesz measure of $w.$
Consider two cases as follows:
 \begin{itemize}
\item Case:$\gamma+q> 1$.
\end{itemize}
Applying the iterated inequality \eqref{iterated} with $t=\gamma+q$, together with Fubini's theorem and H\"{o}lder's inequality with the exponents $\frac{\gamma}{\gamma+q-1}$ and $\frac{\gamma}{1-q},$ we obtain
\begin{align*}
\mathcal{E}_{\frac{\gamma+q}{1-q}}[\omega]=\int_{\Omega}(\mathbf{G}\omega)^{\frac{\gamma+q}{1-q}}d\omega&=\int_{\Omega} (\mathbf{G}\mu)^{\gamma+q}d\omega \\&\leq C\int_{\Omega}\mathbf{G}((\mathbf{G}\mu)^{\gamma+q-1}d\mu)d\omega \\
&\leq C\Big(\int_{\Omega} (\mathbf{G}\mu)^{\gamma}d\mu\Big)^{\frac{\gamma+q-1}{\gamma} }\Big(\int_\Omega (\mathbf{G}\omega)^\frac{\gamma}{1-q}d\mu\Big)^ {\frac{1-q}{\gamma}}\\
&=C \int_{\Omega}(\mathbf{G}\mu)^\gamma d\mu \\ 
&= C \mathcal{E}_\gamma [\mu]
\end{align*} 
\begin{itemize}
\item Case:$\gamma+q\leq1$.
\end{itemize}
Write \[
\int_{\Omega}(\mathbf{G}\mu)^{\gamma+q}d\omega=\int_{\Omega}(\mathbf{G}\mu)^{\gamma+q} F^{a-1} F^{1-a} d\omega,
\] where $a=\gamma+q$ and $F$ is a positive $\omega$-measurable function to be determined later. Applying H\"{o}lder's inequality with the exponents $\frac{1}{a}$ and $\frac{1}{1-a}$, we get
\begin{align*}
\mathcal{E}_{\frac{\gamma+q}{1-q}}[\omega]=\int_{\Omega}(\mathbf{G}\omega)^{\frac{\gamma+q}{1-q}}d\omega
&=\int_{\Omega}(\mathbf{G}\mu)^{\gamma+q} F^{a-1} F^{1-a} d\omega\\
&\leq  \Big( \int_{\Omega}  (\mathbf{G}\mu)^{\frac{\gamma+q}{a}} F^{\frac{a-1}{a}} d\omega \Big)^{a} \Big(\int_{\Omega} F d\omega \Big)^{1-a}
\end{align*}
Setting $F=(\mathbf{G}\omega)^{\frac{\gamma+q}{1-q}}$. 
\be\label{a}
\Big(\mathcal{E}_{\frac{\gamma+q}{1-q}}[\omega]\Big)^{\gamma+q}\leq \Big( \int_{\Omega} (\mathbf{G}\mu)\big(\mathbf{G}\omega \big)^{(\frac{\gamma+q}{1-q})(\frac{\gamma+q-1}{\gamma+q})}d\omega \Big)^{\gamma+q}
\ee
The right-hand side of \eqref{a} is estimated by using Fubini's theorem, followed by inequality \eqref{iteratedgeq} with $t=\frac{\gamma}{1-q}$
\begin{align*}
\mathcal{E}_{\frac{\gamma+q}{1-q}}[\omega]
\leq\int_{\Omega} \mathbf{G} ( \big( \mathbf{G}\omega \big)^{\frac{\gamma+q-1}{1-q}}d\omega)d\mu
&\leq C\int_{\Omega} (\mathbf{G}\omega)^{\frac{\gamma}{1-q}}d\mu\\
&=C\int_{\Omega} (\mathbf{G}\mu)^\gamma d\mu \\ 
&=\mathcal{E}_{\gamma}[\mu].
\end{align*}
This completes the proof of the lemma.
 \end{proof}

 The following lemma gives Green potentials norm estimates in terms of generalized energy.
 
\begin{Lem}\label{lemmaformu}
Let $G$ be a positive Green function associated with $\mathcal{L}$ in $\Omega \subset  \R^n$. Let $\mu \in \M^{+}(\Omega)$ such that $\mathbf{G}\mu\not\equiv\infty.$ Then, for $0<\gamma<\infty,$ 
$\mathbf{G}\mu\in L^{\gamma} (\Omega, d\mu)$  implies $\mathbf{G}\mu\in L^{p}(\Omega,dx)$, i.e.,
\[
\|\mathbf{G}\mu\|_{L^p(\Omega,dx)}\leq c \Big( \mathcal{E}_\gamma [\mu]\Big)^\frac{1}{\gamma+1}
\]
where $p=\frac{n(1+\gamma)}{n-2}$ and $c$ is a positive constant depending on $\gamma$.\end{Lem}
\begin{proof}
 Notice that $w:=(\mathbf{G}\mu)^{1-q}$ with $0<q<1$ is a positive $\mathcal{A}$-superharmonic function on $\Omega$ since $\mathbf{G}\mu\not\equiv +\infty,$ and  $w:=\mathbf{G}\omega,$ where $\omega\in\mathcal{M}^+(\Omega)$ is the Riesz measure of $w,$ see \cite{SV2}.

Applying the Lemma \ref{mainlemma} together with Lemma \ref{lemma2}, we get the desired estimate,
\[
\|\mathbf{G}\mu\|_{L^p(\Omega,dx)}= \big\|\mathbf{G}\omega\big\|^{\frac{1}{1-q}}_{L^{\frac{p}{1-q}}(\Omega,dx)}   \leq \tilde{C} \big\|\mathbf{G}\omega\big\|^{\frac{\gamma+q}{(\gamma+1)(1-q)}}_{L^{\frac{\gamma+q}{1-q}}(\Omega,d\omega)}\leq \tilde{C}C\big\|\mathbf{G}\mu\big\|_{L^\gamma(\Omega,d\mu)}^{\frac{\gamma}{\gamma+1}}.
\]
where $\tilde{C}$ and $C$ are constants in Lemma \ref{lemma2} and Lemma \ref{mainlemma} respectively.
\end{proof}
We are now ready to prove the main theorem of this work.
\begin{proof}[Proof of Theorem \ref{thm:classical}] 
Suppose that \eqref{greenpotentialwithsigma} and \eqref{greenpotentialwithmu} hold for  $\gamma=\frac{p(n-2)-n}{n}.$ The condition \eqref{greenpotwithsigma} is satisfied by Lemma \ref{lemma}. As a result, according to Theorem \ref{thm:solofinteq}, the integral equation
\[
u= \mathbf{G}(u^qd\sigma)+\mathbf{G}\mu\quad\text{in}\quad\Omega
\]
has a positive solution $u\in L^{\gamma+q}(\Omega,d\sigma).$ In order to get solution $u\in L^p(\Omega,dx),$
we combine Lemma \ref{weightedlemma}, Lemma \ref{lemma2} and Lemma \ref{lemmaformu}.
 Then, we find that
\begin{align*}
	\|u\|_{L^{p}(\Omega,dx)} 
	&\leq  \|\mathbf{G}(u^qd\sigma)\|_{L^{p}(\Omega,dx)} + \|\mathbf{G} \mu\|_{L^{p}(\Omega,dx)}\\
	&\leq  C\| \mathbf{G}\sigma\|_{L^{\frac{\gamma+q}{1-q}}(\Omega,d\sigma)} \; \|u^q\|_{L^{\frac{\gamma+q}{q}}(\R^n,\,d\sigma)} 
	+ c\Big(\int_{\Omega}(\mathbf{G}\mu)^\gamma \,d\mu\Big)^{\frac{1}{\gamma+1}}\\
	&= C\|u\|^{q}_{L^{\gamma+q}(\Omega,\,d\sigma)}  <+ \infty.
	\end{align*}
This shows that there exists a positive solution $u \in L^{p}(\Omega, dx)$ to \eqref{classicallaplacian}
\end{proof}

We finish this paper by providing a proof of Corollary \ref{corollary}.  The following proof is mainly influenced by Seesanea and Verbitsky \cite{SV2} and by Boccardo and Orsina \cite{BO}.
\begin{proof}[Proof of Corollary \ref{corollary}] 
  Setting $\gamma=\frac{p(n-2)-n}{n} $. Then $s_1=\frac{np}{n(1-q)+2p}>1$. By H\"{o}lder inequality,
  \be \label{equation}
  \int_\Omega \big(\mathbf{G}\sigma \big)^{\frac{\gamma+q}{1-q}}d\sigma \leq \big\| \mathbf{G}\sigma\big\|^{\frac{\gamma+q}{1-q}} _{L^{(\frac{\gamma+q}{1-q})s'_1}(\Omega,dx) } \| \sigma \|_{L^{s_1}(\Omega,dx)},
  \ee
  where $s'_1=\frac{np}{p(n-2)-n(1-q)}$ is the conjugate of $s_1$.
  We see that
  \[
  \frac{1}{s_1}+\frac{1}{\big(\frac{\gamma+q}{1-q}\big)s'_1}=\frac{2}{n}.
  \]
Appealing to Hardy-Littlewood-Sobolev inequality,
\be\label{equationone}
\begin{split}
\big\| \mathbf{G}\sigma\big\| _{L^{ \big(\frac{\gamma+q}{1-q}\big) s'_1}(\Omega,dx)}\lesssim \big\| \mathbf{G} \tilde{\sigma}\big\|_{L^{ \big(\frac{\gamma+q}{1-q}\big) s'_1}(\mathbb{R}^n,dx)} &\lesssim \|\tilde{\sigma}\|_{L^{s_1}(\mathbb{R}^n,dx)}\\ 
&=\|\sigma\| _{L^{s_1}(\Omega,dx)}.
\end{split}
\ee
Here, $\tilde{\sigma}$ is the zero extension of $\sigma$ to $\mathbb{R}^n$.
Thus, by \eqref{equation} and \eqref{equationone},
\[
\big\| \mathbf{G}\sigma\big\|_{ L^{ \frac{\gamma+q}{1-q} }(\Omega,d\sigma) }
\leq \Big( \big\| \sigma  \big\|^{\frac{\gamma+q}{1-q}+1}_{L^{s_1}(\Omega,dx)}\Big)^{\frac{1-q}{\gamma+q}}
= \big\| \sigma  \big\|^{\frac{\gamma+1}{\gamma+q}}_{L^{s_1}(\Omega,dx)}
<+\infty
\]
Hence, \eqref{greenpotentialwithsigma} is valid.
Similarly, we note that $s_2=\frac{np}{n+2p}>1$. By H\"{o}lder inequality,
\be \label{equationtwo}
\int_\Omega (\mathbf{G}\mu)^\gamma d\mu \leq \| \mathbf{G}\mu\|^{\gamma}_{ L^{\gamma s'_2}(\Omega,dx) } \|\mu\|_{L^{s_2}(\Omega,dx)}
\ee
where $s'_2=\frac{s_2}{s_2 -1}$. We see that
\[
\frac{1}{s_2}+\frac{1}{\gamma s'_2}=\frac{2}{n}.
\]
Taking  Hardy-Littlewood-Sobolev inequality,
\be\label{equationthree}
\| \mathbf{G}\mu\| _{L^{ \gamma s'_2}(\Omega,dx)}\lesssim \| \mathbf{G}  \tilde{\mu}\|_{L^{ \gamma s'_2}(\mathbb{R}^n,dx)}\lesssim \|\tilde{\mu}\|_{L^{s_2}(\mathbb{R}^n,dx)}=\|\mu\| _{L^{s_2}(\Omega,dx)},
\ee
where $\tilde{\mu}$ is the zero extension of $\mu$ to $\mathbb{R}^n$.
Thus, by \eqref{equationtwo} and \eqref{equationthree},
\[
\| \mathbf{G}\mu\|_{L^\gamma (\Omega,d\mu)}\leq \big( \|\mu\|^{\gamma+1}_{L^{\gamma}(\Omega,dx)} \big)^{\frac{1}{\gamma}}<+\infty.
\]
Therefore,  \eqref{greenpotentialwithmu} is fulfilled. Consequently, Theorem \ref{thm:classical} yields the existence of 
the minimal positive solution $u \in L^{p}(\Omega, dx)$ to \eqref{classicallaplacian}.
\end{proof}


\section*{Acknowledgments} This study was supported by Thammasat University Research Fund, Contract No. TUFT 52/2566. A.C.M.  gratefully acknowledges financial support from the Excellent Foreign Student (EFS) scholarship,  Sirindhorn International Institute of Technology (SIIT),  Thammasat University.


\bibliographystyle{abbrv} 
\bibliography{reference(MS1)}
\end{document}